\documentclass{article}
\usepackage{amsmath}
\usepackage{amsthm}
\usepackage{amssymb}

\newtheorem{thm}{Theorem}

\newtheorem{defn}{Definition}
\newtheorem{lem}{Lemma}
\newtheorem{prop}{Proposition}
\newtheorem{cor}{Corollary}
\newtheorem{example}{Example}

\newtheorem{remark}{Remark}

\newcommand{\al}{\alpha}
\newcommand{\be}{\beta}

\newcommand{\la}{\lambda}
\newcommand{\eps}{\varepsilon}

\makeatletter
\providecommand*{\shuffle}{%
  \mathbin{\mathpalette\shuffle@{}}%
}
\newcommand*{\shuffle@}[2]{%
  \sbox0{$#1\vcenter{}$}%
  \kern .15\ht0 % side bearing
  \rlap{\vrule height .25\ht0 depth 0pt width 2.5\ht0}%
  \raise.1\ht0\hbox to 2.5\ht0{%
    \vrule height 1.75\ht0 depth -.1\ht0 width .17\ht0 %
    \hfill
    \vrule height 1.75\ht0 depth -.1\ht0 width .17\ht0 %
    \hfill
    \vrule height 1.75\ht0 depth -.1\ht0 width .17\ht0 %
  }%
  \kern .15\ht0 % side bearing
}
\makeatother

\usepackage{tikz, tikz-3dplot, pgfplots}
\usepackage{bm, relsize}
\usepackage{tkz-graph}
\usetikzlibrary[positioning,patterns] % tikz libraries for relative positioning and silly patterns

\usepackage{pgf}
\usepackage{mathrsfs}
\usetikzlibrary{arrows}

\pgfplotsset{compat=1.8}

\tikzstyle{bsq}=[rectangle, draw, thick, minimum width=1cm, minimum height=1cm] 
\tikzstyle{bver}=[rectangle, draw, thick, minimum width=1cm, minimum height=2cm]
\tikzstyle{bhor}=[rectangle, draw, thick, minimum width=2cm, minimum height=1cm]

\tikzstyle{bsqg}=[rectangle, draw=gray!30!white, thick, fill=gray!30!white, minimum width=1cm, minimum height=1cm] 
\tikzstyle{bverg}=[rectangle, draw=gray!30!white, thick, fill=gray!30!white, minimum width=1cm, minimum height=2cm]
\tikzstyle{bhorg}=[rectangle, draw=gray!30!white, thick, fill=gray!30!white, minimum width=2cm, minimum height=1cm]

\usepackage{lipsum}

\title{On the product of generating functions for domino and bi-tableaux}

\author{Ekaterina A. Vassilieva}
\usepackage[backend=bibtex]{biblatex}
\begin{document}

\maketitle
%% note that you DO NOT have to put your abstract here -- it is generated by \maketitle and the \abstract and \resume commands above
\noindent {\bf Abstract.} The connection between the generating functions of various sets of tableaux and the appropriate families of quasisymmetric functions is a significant tool to give a direct analytical proof of some advanced bijective results and provide new combinatorial formulas. In this paper we focus on two kinds of type B Littlewood-Richardson coefficients and derive new formulas using weak composition quasisymmetric functions and Chow's quasisymmetric functions. In the type A case these coefficients give the multiplication table for Schur functions, i.e. the generating functions for classical Young tableaux. We look at their generalisations involving a set of bi-tableaux and domino tableaux.

\section{Type A quasisymmetric functions}
%%%%%%%%%%%%%%%%%%%%%%%%%%%%%%%%%%%%%%%%%%%%%%%%%%%%%%%%%%
%%%%%%%%%%%%%%%%%%%%%%%%%%%%%%%%%%%%%%%%%%%%%%%%%%%%%%%%%%
\subsection{Permutations, Young tableaux and descent sets}
For any positive integer $n$ write $[n] = \{1,\cdots, n\}$, $S_n$ the symmetric group on $[n]$. A  {\bf partition} $\la$ of  $n$, denoted $\la \vdash n$  is a sequence $\la=(\la_1,\la_2,\cdots,\la_p)$ of $\ell(\la)=p$ parts sorted in decreasing order such that $|\la| = \sum_i{\la_i} = n $. A partition $\la$ is usually represented as a~Young diagram of $n=|\la|$ boxes arranged in $\ell(\la)$ left justified rows so that the $i$-th row from the top contains $\la_i$ boxes.  A Young diagram whose boxes are filled with positive integers such that the entries are increasing along the rows and strictly increasing down the columns is called a {\bf semistandard Young tableau}. If the entries of a semistandard Young tableau are restricted to the elements of $[n]$ and strictly increasing along the rows, we call it a {\bf standard Young tableau}. The partition $\la$ is the {\bf shape} of the tableau and we denote $SYT(\la)$ (resp. $SSYT(\la)$) the set of standard (resp. semistandard) Young tableaux of shape $\la$.\\
One important feature of a permutation $\pi$ in $S_n$ is its {\bf descent set}, i.e. the subset of $[n-1]$ defined as $Des(\pi) = \{1\leq i \leq n-1\mid \pi(i)>\pi(i+1)\}$. Similarly, define the {\bf descent set of a standard Young tableau} $T$ as $Des(T) = \{1\leq i \leq n-1\mid i $ is in a strictly higher row than $i+1\}$.
%%%%%%%%%%%%%%%%%%%%%%%%%%%%%%%%%%%%%%%%%%%%%%%%%%%%%%%%%%
\subsection{Quasisymmetric expansion of Schur polynomials}
Let $X = \{x_1, x_2, \cdots \}$ and $Y = \{y_1, y_2, \cdots \}$ be two alphabets of commutative indeterminates. Denote also $XY$ the product alphabet $\{x_iy_j\}_{i,j}$ ordered by the lexicographical order.  Let $F_I(X)$ be the {\bf Gessel's fundamental quasisymmetric function} on $X$ indexed by the subset $I \subset [n-1]$:
\begin{equation}
F_{I}(X) = \sum\limits_{\substack{i_1 \leq \cdots \leq i_n\\k\in I \Rightarrow i_k<i_{k+1}}} x_{i_1}x_{i_2} \cdots x_{i_n}.
\end{equation}
The power series $F_I(X)$ is not symmetric in $X$ but verifies the property that for any strictly increasing sequence 
of indices $i_1 < i_2 <\cdots< i_p$ the coefficient of $x_1^{k_1}x_2^{k_2}\cdots x_p^{k_p}$ is equal to the coefficient of $x_{i_1}^{k_1}x_{i_2}^{k_2}\cdots x_{i_p}^{k_p}$.
Let $s_\la(X)$ be the Schur polynomial indexed by partition $\la$ on alphabet $X$. Schur polynomials are the generating functions for semistandard Young tableaux. As a result,
\begin{equation}
\label{eq : ST}
s_{\la}(X) = \sum_{T \in SSYT(\la)}X^T = \sum_{T \in SYT(\la)}F_{Des(T)}(X).
\end{equation}
Using for any $\pi \in S_n$ Gessel's formula for the coproduct of fundamental quasisymmetric functions in \cite{Ges84} 
\begin{equation}
\label{eq : copr}
F_{Des(\pi)}(XY) = \sum_{\sigma, \rho \in S_n;\; \sigma\rho = \pi}F_{Des(\sigma)}(X)F_{Des(\rho)}(Y),
\end{equation}
and the Cauchy formula for Schur polynomials
\begin{equation}
\label{eq : CauchyPrime}
\sum_{\la \vdash n}s_{\la}(X)s_{\la}(Y) = s_{(n)}(XY),
\end{equation}
one can derive
\begin{equation}
\label{eq : CauchyDes}
\sum_{\substack{\la \vdash n\\T,U \in SYT(\la)}}F_{Des(T)}(X)F_{Des(U)}(Y) = F_{\emptyset}(XY) = \sum_{\pi \in S_n}F_{Des(\pi)}(X)F_{Des(\pi^{-1})}(Y).
\end{equation}
This gives a direct analytical proof of one important property usually proved using the {\it Robinson-Schensted}  correspondence, i.e. that permutations $\pi$ in $S_n$ are in bijection with pairs of standard Young tableaux $T,U$ of the same shape such that $Des(T) = Des(\pi)$ and $Des(U) = Des(\pi^{-1})$.
%%%%%%%%%%%%%%%%%%%%%%%%%%%%%%%%%%%%%%%%%%%%%%%%
\subsection{Product of Schur polynomials}
\label{sec : prodSchur}
%%%%%%%%%%%%%%%%%%%%%%%%%%%%%%%%%%%%%%%%%%%%%%%%
One may write a permutation $\al \in S_n$ as a word on the letters in $[n]$, $\al = \al_1 \al_2 \cdots \al_n$ where $\al_i = \al(i)$. Given $\alpha \in S_n$ and $\beta \in S_m$, let $\alpha \shuffle \beta$ be the set of permutations $\gamma \in S_{m+n}$ obtained by shuffling the letters $\alpha$ and $\overline{\beta} = n+\beta_1\,\, n+\beta_2\cdots\,\, n+\beta_m$ such that the initial order of the letters of $\alpha$ (resp. of $\overline{\beta}$) is preserved in $\gamma$. By extension if $A$ and $B$ are two sets of permutations, we denote $A\shuffle B$ the set of all the shuffles of a permutation of $A$ and a permutation of $B$.  We have the product formula 
\begin{equation}
\label{eq : ProdQuasi}
F_{Des(\alpha)}(X)F_{Des(\beta)}(X) =  \sum_{\gamma \in \alpha \shuffle \beta}F_{Des(\gamma)}(X).
\end{equation}
Given a standard Young tableau $T$ of shape $\la$, Equation \eqref{eq : CauchyDes} proves the existence of classes of permutations $C_{T}$ such that for any $\pi \in C_T$, $Des(\pi^{-1}) = Des(T)$ and there is a descent preserving bijection between $C_T$ and $SYT(\la)$. 
\begin{remark}
Classically the $C_T$ are the {\it Knuth classes} of permutations but here we derive these results directly without using plactic relations and/or the Robinson-Schensted correspondence. 
\end{remark}
\noindent Equation \eqref{eq : ProdQuasi} directly implies for integer partitions $\la$ and $\mu$ and $T, U \in SYT(\la)\times SYT(\mu)$:
\begin{equation}
\label{eq : ProdSchur}
s_{\la}(X)s_{\mu}(X) = \sum_{\al \in C_T}F_{Des(\al)}(X)\sum_{\beta \in C_U}F_{Des(\beta)}(X) = \sum_{\gamma \in C_T \shuffle C_U}F_{Des(\gamma)}(X).
\end{equation}
For any $\nu \vdash |\la| + |\mu|$ define formally the coefficients $c^\nu_{\la \mu}$ as the coefficients of the expansion of $s_{\la}(X)s_{\mu}(X)$ in the Schur basis $(s_\nu)_\nu$. Equation \eqref{eq : ProdSchur} shows for some $V_\nu \in SYT(\nu)$ that 
\begin{equation}
\label{eq : LR}
\sum_{\nu \vdash  |\la| + |\mu|} c^\nu_{\la \mu} s_{\nu}(X) = \sum_{\substack{\nu \vdash  |\la| + |\mu| \\ \gamma \in C_{V_\nu}}}c^\nu_{\la \mu}F_{Des(\gamma)}(X) = \sum_{\gamma \in C_T \shuffle C_U}F_{Des(\gamma)}(X).
\end{equation}
As a result, we have the following proposition.
\begin{prop}
\label{prop : typeA}
For $I \subset [|\la|+|\mu|-1]$ denote $D_I$ the set of permutations $\pi$ such that $Des(\pi) =I$. Using the notations above for integer partitions $\la, \mu$ and for $T,U$ in $SYT(\la)\times SYT(\mu)$, there is a set of tableaux $V_\nu \in SYT(\nu)$ indexed by the set of integer partitions $\nu \vdash |\la| + |\mu|$ such that 
\begin{equation}
\label{eq : lem1}
\sum_{\nu \vdash  |\la| + |\mu|}c^\nu_{\la \mu}|C_{V_\nu} \cap D_I| =  |\left (C_T \shuffle C_U \right) \cap D_I|.
\end{equation}
\end{prop}
The following sections are dedicated to the extension of this method to weak composition quasisymmetric functions and Chow's type B quasisymmetric functions. 
%%%%%%%%%%%%%%%%%%%%%%%%%%%%%%%%%%%%%%%%%%%%%%%%%%%%%%%%%%%%%%%%%%%%%%%%%%%%%%%
%%%%%%%%%%%%%%%%%%%%%%%%%%%%%%%%%%%%%%%%%%%%%%%%%%%%%%%%%%%%%%%%%%%%%%%%%%%%%%%
\section{Type B extension to weak composition quasisymmetric functions}
%%%%%%%%%%%%%%%%%%%%%%%%%%%%%%%%%%%%%%%%%%%%%%%%%%%%%%%%%%
%%%%%%%%%%%%%%%%%%%%%%%%%%%%%%%%%%%%%%%%%%%%%%%%%%%%%%%%%%
\subsection{Weak composition quasisymmetric functions and \break coloured permutations}
Weak composition quasisymmetric functions are a generalisation of Gessel's quasisymmetric functions introduced in \cite{GuoYuZha17} and \cite{GuoYuThi19} in order to provide a framework for the study of a question a G.-C. Rota relating {\it Rota-Baxter algebras} and symmetric type functions.\\
We consider a finite alphabet $X = \{x_1, x_2,\, \ldots , x_M\}$. For any non-negative integer $n$, let $\mathcal{WC}(n)$ be the set of {\bf weak compositions} of $n$, i.e. the set of ordered tuples of non-negative integers that sum up to $n$. A weak composition $\al \in \mathcal{WC}(n)$ is of the form $$\al =(0^{i_1}, s_1, 0^{i_2}, s_2,\ldots, 0^{i_k}, s_k, 0^{i_{k+1}})$$
with $i_j \geq 0$, $1\leq j \leq k+1$, $s_i >0$, $1\leq j \leq k$ and $\sum s_i = n$. We further define $\ell(\al)=k+i_1+\cdots+i_{k+1}$ the
{\bf length} of $\al$, $\ell_0(\al)= i_1 +\cdots+i_{k+1}$ as its {\bf $0$-length}, $|\al| = n$ the {\bf weight} of $\al$ and $||\al|| = |\al| + \ell_0(\al)$ the {\bf total weight} of $\al$. The descent set of $\al$ is
$$D(\al)=\{a_1,\ldots, a_k | a_j =i_1+s_1+\cdots+i_j+s_j, j=1,\ldots,k\}.$$
\begin{defn}
Given a weak composition $\al =(0^{j_1}, s_1, 0^{j_2}, s_2,\ldots, 0^{j_k}, s_k, 0^{j_{k+1}})$ with $D(\al) = \{a_1, a_2,\ldots, a_k\}$ define the {\bf weak composition fundamental quasisymmetric function} indexed by $\al$ the formal power series
\footnotesize
\begin{align}
\label{eq : defWCF}
\nonumber &\widetilde{F}_\al(X) = \\
\hspace{-0mm}&\sum_{\substack{1\leq i_1 \leq \cdots \leq i_{a_k+j_{k+1}} \\ j\in D(\al) \Rightarrow i_j < i_{j+1}}}\hspace{-6mm}x_{i_1}^0\cdots x_{i_{j_1}}^0x_{i_{j_1+1}}\cdots x_{i_{a_1}}\cdots x_{i_{a_{k-1}+j_k}+1}\cdots x_{i_{a_k}} x_{i_{a_k+1}}^0\cdots x_{i_{a_{k}+j_{k+1}}}^0
\end{align}
\end{defn}
\normalsize
Weak composition fundamental quasisymmetric functions may also be described in terms of $P$-partitions indexed by {\it coloured permutations}. Define the total order $$\overline{1}<\overline{2}<\overline{3}<\cdots<1<2<3\cdots$$
and denote $|\overline{i}| = i$ for $i>0$. We look at chain posets $P_\pi = \{ \pi_1 <_{P_\pi}\pi_2<_{P_\pi}\cdots<_{P_\pi}\pi_l\}$ indexed by coloured permutations $\pi$ for any non-negative integer $l$ where the $\pi_i$ are distinct letters in $[l]\cup[\overline{l}]$ and such that $|\pi_1||\pi_2|\cdots|\pi_l|$ has no repeated letters. A {\bf $ P_\pi$-partition} is a map $f\;:\;P_\pi \rightarrow [M]$ such that for $x <_{P_\pi} y$, $f(x) \leq f(y)$ with $x>y \Rightarrow f(x) < f(y)$. The {\bf weight} of $f$, $w(f)$ is the monomial
$$w(f) = \prod_{\pi_i \in P_\pi}x_{f(\pi_i)}^{\delta_{\pi_i}},\mbox{ where } \delta_{\pi_i} = \begin{cases}0\mbox{ if }\pi_i\mbox{ is an overlined letter,}\\ 1\mbox{ otherwise.}\end{cases}$$
Denote $\mathcal{A}(\pi)$ the set of $ P_\pi$-partitions. We are interested in the generating functions $\Gamma$ indexed by coloured permutations $\pi$
\begin{equation}
\Gamma(\pi) = \sum_{f \in \mathcal{A}(\pi)}w(f).
\end{equation}
Li shows in \cite{Li18} that the $\Gamma(\pi)$ can be expressed as a $\mathbb{Z}$-linear combination of the $F_\al$ but not always positive. In particular if the indices $j_1<j_2<\cdots<j_p$ of the overlined letters in a coloured permutation $\pi$ are such that $\pi_{j_1}<\pi_{j_2}<\cdots<\pi_{j_p}$, denote the weak composition $$comp(\pi) = (0^{i_1}, s_1, 0^{i_2}, s_2,\ldots, 0^{i_k}, s_k, 0^{i_{k+1}})$$ where $(s_1, s_2, \ldots, s_k)$ is the tuple of lengths of the sequences of consecutive increasing non-overlined letters in $\pi$ and $i_r$ is the quantity of overlined letters between sequences $r-1$ and $r$. We have $\ell_0(comp(\pi))=p$ and we denote $S_n^{(p)}$ the set of such coloured permutations with $|comp(\pi)| = n$. As an example for $\pi = 597\overline{1}\overline{2}6\overline{3}8\overline{4}$, we have $comp(\pi) = (2,1,0^2,1,0,1,0)$. As a direct consequence of Equation \eqref{eq : defWCF}, for $n,p\geq0$ and $\pi \in S_n^{(p)}$, one has (\cite{Li18}) $$\widetilde{F}_{comp(\pi)} = \Gamma(\pi).$$ On the other hand, take $\pi = 3\overline{2}\overline{1}$, then it is easy to show that $\Gamma(\pi) = \widetilde{F}_{(1,0^2)} - \widetilde{F}_{(1,0)}.$
Finally, as in Section \ref{sec : prodSchur}, define for $\alpha, \beta \in S_n^{(p)} \times S_r^{(q)}$ the set $\al \shuffle \beta$ of coloured permutations obtained by shuffling the letters of $\widetilde{\al}$ and $\widehat{\be}$ where the non-overlined letters of $\al$ are shifted by $q$ to get $\widetilde{\al}$ and the overlined (resp. non-overlined) letters of $\beta$ are shifted by $p$ (resp. by $p+n$) to get $\widehat{\be}$. For instance, let $\al = 32\overline{1}$, $\beta = \overline{1}2$, then $\widetilde{\al} = 43\overline{1}$, $\widehat{\be} = \overline{2}5$ and  $$\al \shuffle \beta =  43\overline{1}\overline{2}5 + 43\overline{2}\overline{1}5+4\overline{2}3\overline{1}5+\overline{2}43\overline{1}5+43\overline{2}5\overline{1}+4\overline{2}35\overline{1}+\overline{2}435\overline{1}+4\overline{2}53\overline{1}+\overline{2}453\overline{1}+\overline{2}543\overline{1}.$$
One can see that despite the fact that $\alpha$ and $\beta$ belong to $S_n^{(p)}$ and $S_r^{(q)}$, not all the permutations in $\al \shuffle \beta$ fulfil the conditions of proper order of the overlined letters. We have for $\alpha, \beta \in S_n^{(p)} \times S_r^{(q)}$,
\begin{equation}
\widetilde{F}_{comp(\al)}(X)\widetilde{F}_{comp(\beta)}(X) = \sum_{\gamma \in \al \shuffle \beta}\Gamma (\gamma).
\end{equation}
In particular for $n=r=0$
\begin{equation}
\label{eq : FFF}
\widetilde{F}_{(0^p)}(X)\widetilde{F}_{(0^q)}(X) = \sum_{j=0}^p(-1)^j\binom{p}{j}\binom{p+q-j}{p}\widetilde{F}_{(0^{p+q-j})}(X).
\end{equation}
In some sense, weak composition fundamental quasisymmetric functions are a specialisation of {\it Poirier quasisymmetric functions} (see \cite{Poi98}), a classical type B generalisation of Gessel's quasisymmetric functions that we do not detail here. In \cite{AdiAthEliRoi15}, the authors provide a connexion between the generating functions for Young {\it bi-tableaux} and the fundamental quasisymmetric functions of Poirier. As a result, there should be a natural bi-tableaux interpretation of weak composition quasisymmetric functions. This is the focus of the next section. 

 \subsection{Tableau-theoretic interpretation}
Let $\la$ be an integer partition and $p$ a non-negative integer. We consider pairs of Young tableau $(T^-, T^+)$ where $T^-$ is of shape $(p)$ and $T^+$ has shape $\la$. We call such a bi-tableau a $(p)$-tableau. We proceed with the formal definition. 

\begin{defn}
For any couple of non-negative integer $p,n$ and an integer partition $\la \vdash n$, a {\bf semistandard $(p)$-tableau} $T=(T^-, T^+)$ is a pair of Young diagrams of bi-shape $((p), \la)$ whose boxes are filled with the integers in $[M]$ such that the entries are strictly increasing down the columns and non-decreasing along the rows. The {\bf shape} of $T$ is noted $sh(T) = ((p), \la)$ and we denote $SSBT^{(p)}(\la)$ the set of all $(p)$-tableaux of shape $((p), \la)$. If the boxes are filled with integers in $[n+p]$ all used exactly once, we call $T$ a {\bf standard $(p)$-tableau} and we denote  $SBT^{(p)}(\la)$ the set of such tableaux.
\end{defn}
\noindent We also need the following notions. The {\bf standardisation} $T^{st}$ of a semistandard $(p)$-tableau $T$ is the standard $(p)$-tableau of shape $((p), \la)$ obtained by relabelling the boxes of $T$ with all the integers in $[|\la|+p]$ such that the order of the labels is preserved. When a label in $T$ is used both in $T^-$ and $T^+$, the boxes of $T^-$ are relabelled first. Finally, given a standard $(p)$-tableau $T \in SBT^{(p)}(\la)$ with $|\la| = n$ denote $comp(T)$ the weak composition of $n$,  $comp(T) = (0^{i_1}, s_1, 0^{i_2}, s_2,\ldots, 0^{i_k}, s_k, 0^{i_{k+1}})$, where $i_1$ is the number of consecutive labels $k=1,2\ldots,i_1$ in $T^-$ ($i_1 = 0$ if $1$ is in $T^+$). Then $s_1$ is the number of consecutive labels $k = i_1+1, \cdots, i_1+s_1$ in $T^+$ such that no label $k$ is in a higher row than $k+1$. Then $i_2$ is the number of consecutive labels $k=i_1+s_1+1,2\ldots,i_1+s_1+i_2$ in $T^-$, etc...

\begin{example} 
The following semistandard $(3)$-tableau belongs to $SSBT^{(3)}((2,2,1))$.
$$
T = \left(T^- =
\begin{matrix} 
\resizebox{!}{0.6cm}{%
\begin{tikzpicture}[node distance=0 cm,outer sep = 0pt]
	      \node[bsq] (1) at (0,  0) {\bf \Large 2};
	      \node[bsq] (2) [right = of 1] {\bf \Large 4};
	      \node[bsq] (3) [right = of 2] {\bf \Large 4};	
\end{tikzpicture}
}
\end{matrix},~ 
T^+ =
\begin{matrix} 
\resizebox{!}{1.8cm}{%
\begin{tikzpicture}[node distance=0 cm,outer sep = 0pt]
	      \node[bsq] (1) at (0,  0) {\bf \Large 1};
	      \node[bsq] (2) [right = of 1] {\bf \Large 4};
	      \node[bsq] (3) [below = of 1] {\bf \Large 5};
	      \node[bsq] (4) [right = of 3] {\bf \Large 5};
	      \node[bsq] (5) [below = of 3] {\bf \Large 9};
\end{tikzpicture}
}
\end{matrix}
\right)
$$
Furthermore, its standardisation is:
$$
T^{st} = \left(T^- =
\begin{matrix} 
\resizebox{!}{0.6cm}{%
\begin{tikzpicture}[node distance=0 cm,outer sep = 0pt]
	      \node[bsq] (1) at (0,  0) {\bf \Large 2};
	      \node[bsq] (2) [right = of 1] {\bf \Large 3};
	      \node[bsq] (3) [right = of 2] {\bf \Large 4};	
\end{tikzpicture}
}
\end{matrix},~ 
T^+ =
\begin{matrix} 
\resizebox{!}{1.8cm}{%
\begin{tikzpicture}[node distance=0 cm,outer sep = 0pt]
	      \node[bsq] (1) at (0,  0) {\bf \Large 1};
	      \node[bsq] (2) [right = of 1] {\bf \Large 5};
	      \node[bsq] (3) [below = of 1] {\bf \Large 6};
	      \node[bsq] (4) [right = of 3] {\bf \Large 7};
	      \node[bsq] (5) [below = of 3] {\bf \Large 8};
\end{tikzpicture}
}
\end{matrix}
\right)
$$
Finally the weak composition of $T^{st}$ is $comp(T^{st}) = (1,0^3,1,2,1)$.
\end{example}

\noindent For $\la \vdash n$ and $p \geq 0$ denote $s_\la^{(p)}$ the generating function for such tableaux. We have

\begin{equation}
\label{eq : genfuncbt}
s_\la^{(p)}(X)= \sum_{T \in SSBT^{(p)}(\la)}X^{T^+}.
\end{equation}

\noindent Recall the notation of the Schur symmetric function $s_\la$ indexed by $\la$. The following lemma is immediate. 
\begin{lem}
For any $\la \vdash n$ and $p \geq 0$, one has
\begin{equation}
s_\la^{(p)}(X)= \binom{M+p-1}{p}s_\la(X).
\end{equation}
\end{lem}

\noindent We can now state the main result of this section.
\begin{prop}
\label{prop : stf}
For non-negative integers $p$ and $n$ and an integer partition $\la \vdash n$ the generating function for $(p)-tableau$ of shape $((p), \la)$ is related to weak composition fundamental quasisymmetric functions through
\begin{equation}
s_\la^{(p)}(X)= \sum_{T \in SBT^{(p)}(\la)}\widetilde{F}_{comp(T)}(X).
\end{equation}
\end{prop}
\begin{proof}(sketch)
Starting from Equation \eqref{eq : genfuncbt}, sort all the semi-standard $(p)$-tableaux according to their standardisation. Then, given a standard $(p)$-tableau $T_0$ show that $$\widetilde{F}_{comp(T_0)}(X) = \sum_{T \in SSBT^{(p)}(\la); T^{st} = T_0} X^{T^+}.$$
Finally, sum on $SBT^{(p)}(\la)$ to get the result.
\end{proof}
\subsection{Product formulas}
We use Proposition \ref{prop : stf} to derive new results. 
We show the following formula.
\begin{thm}
\label{thm : mainWC}
Let $p,q,n,m$ be non-negative integers and $\la, \mu$ integer partitions of $n$ and $m$. For $\nu \vdash |\la| + |\mu|$ recall the coefficients $c^\nu_{\la, \mu}$ introduced in Section \ref{sec : prodSchur}. The following identity holds.
\begin{equation}
\sum_{\substack{T\in SBT^{(p)}(\la)\\ U \in SBT^{(q)}(\mu)}}\widetilde{F}_{comp(T)}(X)\widetilde{F}_{comp(U)}(X) = \hspace{-8mm}\sum_{\substack{0\leq j \leq p\\ \nu \vdash |\la|+|\mu| \\ V \in SBT^{(p+q-j)}(\nu)}}\hspace{-8mm}(-1)^j\binom{p}{j}\binom{p+q-j}{p}c^{\nu}_{\la\mu}\widetilde{F}_{comp(V)}(X).
\end{equation}
\end{thm}
\begin{proof}(sketch)
First, note that for any integer $p$ $$s^{(p)}_\emptyset = \widetilde{F}_{(0^p)} = \binom{M+p-1}{p}.$$ Then use Equation \eqref{eq : FFF} and the definition of the coefficients $c^{\nu}_{\la\mu}$ to write
\begin{align*}
s^{(p)}_\la s^{(q)}_\mu &= \sum_{\nu \vdash |\la|+|\mu|}c^\nu_{\la \mu} \binom{M+p-1}{p}\binom{M+q-1}{q}s_\nu,\\
&=\sum_{\nu \vdash |\la|+|\mu|}c^\nu_{\la \mu}\sum_{j=0}^p(-1)^j\binom{p}{j}\binom{p+q-j}{p}\binom{M+p+q-j-1}{p+q-j}s_\nu,\\
&=\sum_{\nu \vdash |\la|+|\mu|}c^\nu_{\la \mu}\sum_{j=0}^p(-1)^j\binom{p}{j}\binom{p+q-j}{p}s^{(p+q-j)}_\nu,
\end{align*}
and apply Proposition \ref{prop : stf} to get the result. 
\end{proof}
There is an interesting corollary of Theorem \ref{thm : mainWC} in terms of generating functions for $P_\pi$-partitions. For a (classical) standard Young tableau $T$, recall the permutation class $C_T$ of Section \ref{sec : prodSchur}. Given a non-negative integer $p$, we denote $\eps_p$ the coloured permutation $\eps_p = \overline{12\cdots p}$. The result can be stated as follows.
\begin{cor} For non-negative integers $p,q$, integers partitions $\la$ and $\mu$ and $T \in SYT(\la)$, $U \in SYT(\mu)$, one has:
\begin{equation}
\sum_{\gamma \in (\eps_p \shuffle C_T) \shuffle (\eps_q \shuffle C_U)}\Gamma(\gamma) = \hspace{-0mm}\sum_{\substack{0\leq j \leq p\\ \nu \vdash |\la|+|\mu| \\ V \in SBT^{(p+q-j)}(\nu)}}\hspace{-0mm}(-1)^j\binom{p}{j}\binom{p+q-j}{p}c^{\nu}_{\la\mu}\widetilde{F}_{comp(V)}(X).
\end{equation}
\end{cor}
\begin{proof}
We omit the proof in this extended abstract.
\end{proof}
%\begin{equation}
%\binom{M+p-1}{p}s^{(p)}_{(n)}(XY) = \binom{M+p-1}{p}\binom{MN+p-1}{p}\sum_{\la}s_\la(X)s_\la(Y)= \binom{MN+p-1}{p}\sum_{\la}s_\la^{(p)}(X)s_\la(Y)
%\end{equation}

%%%%%%%%%%%%%%%%%%%%%%%%%%%%%%%%%%%%%%%%%%%%%%%%%%%%%%%%%%%%%%%%%%%%%%%%%%%%%%%
%%%%%%%%%%%%%%%%%%%%%%%%%%%%%%%%%%%%%%%%%%%%%%%%%%%%%%%%%%%%%%%%%%%%%%%%%%%%%%%
\section{Type B extension to Chow's quasisymmetric functions}
%%%%%%%%%%%%%%%%%%%%%%%%%%%%%%%%%%%%%%%%%%%%%%%%%%%%%%%%%%
%%%%%%%%%%%%%%%%%%%%%%%%%%%%%%%%%%%%%%%%%%%%%%%%%%%%%%%%%%
\subsection{Signed permutations and domino tableaux}
Let $B_n$ be the {\bf hyperoctahedral group} of order $n$. $B_n$ is composed of all signed permutations $\pi$ on $\{\text{-}n, \cdots,\text{-}2, \text{-}1, 0, 1, 2, \cdots, n\}$ such that $\pi(-i) = -\pi(i)$ for all $i$. The {\bf total colour} of $\pi \in B_n$ is the number $tc(\pi) = |\{i\geq 1; \pi(i)<0\}|$. Its {\bf descent set} is the subset of $\{0\}\cup[n-1]$ equal to $Des(\pi) = \{0 \leq i \leq n-1 \mid \pi(i) > \pi(i+1)\}$. The main difference with respect to the case of the symmetric group is the possible descent in position $0$ when $\pi(1)$ is a negative integer.\\ 
For $\la \vdash 2n$, a {\bf standard domino tableau} $T$ of shape $\la$ is a Young diagram of shape $shape(T)=\la$ tiled by {\bf dominoes}, i.e. $2\times1$ or $1\times2$ rectangles filled with the elements of $[n]$ such that the entries are strictly increasing along the rows and down the columns. In the sequel we consider only the set $\mathcal{P}^0(n)$ of {\it empty $2$-core partitions} $\la \vdash 2n$ that fit such a tiling. A standard domino tableau $T$ has a descent in position $i>0$ if $i+1$ lies strictly below $i$ in $T$ and has descent in position $0$ if the domino filled with $1$ is vertical. We denote $Des(T)$ the set of all its descents. A {\bf semistandard domino tableau} T of shape $\la \in \mathcal{P}^0(n)$ and weight $w(T) = \mu =(\mu_0, \mu_1, \mu_2,\cdots)$ with $\mu_i \geq 0$ and $\sum_i \mu_i = n$ is a tiling of the Young diagram of shape $\la$ with horizontal and vertical dominoes labelled with integers in $\{0,1,2,\cdots\}$ such that labels are non decreasing along the rows, strictly increasing down the columns and exactly $\mu_i$ dominoes are labelled with $i$. 
If the top leftmost domino is vertical, {\bf it cannot be labelled $\bf 0$}. Denote $SDT(\la)$ (resp. $SSDT(\la)$) the set of standard (resp. semistandard) domino tableaux of shape $\la$. 
Finally, we denote $sp(T)$, the {\bf spin} of (semi-) standard domino tableau $T$, i.e. {\bf half the number} of its vertical dominoes. 
\begin{remark}
\label{rem : zero}
The possible labelling of top leftmost horizontal dominoes in semistandard domino tableaux with $0$ was first introduced by the authors in \cite{MayVas17} and is essential to connect their generating functions with Chow's quasisymmetric functions.
\end{remark}
\begin{figure} [h]
\begin{center}
$$
T_1 =
\begin{matrix}
\resizebox{!}{2.7 cm}{%
\begin{tikzpicture}[node distance=0 cm,outer sep = 0pt]
        \node[bver] (1) at ( 0,   0) {\bf \Large 1};
        \node[bver] (2) [right = of 1] {\bf \Large 2};
        \node[bver] (3) [right = of 2] {\bf \Large 3};         
        \node[bhor] (4) at ( 0.5,   -1.5) {\bf \Large 4}; 
        \node[bhor] (5) at ( 3.5,   0.5) {\bf \Large 5};        
        \node[bhor] (6) [below = of 5] {\bf \Large 6};      
        \node[bver] (7) at ( 0,   -3) {\bf \Large 7};       
        \node[bhor] (8) [right = of 4] {\bf \Large 8};
\end{tikzpicture}  
}
\end{matrix}
~~~~~~~
T_2 = 
\begin{matrix}
\resizebox{!}{2.7 cm}{%
\begin{tikzpicture}[node distance=0 cm,outer sep = 0pt]
        \node[bver] (1) at ( 0,   0) {\bf \Large 1};
        \node[bver] (2) [right = of 1] {\bf \Large 2};
        \node[bhor] (3) at ( 2.5,   0.5) {\bf \Large 3};           
        \node[bhor] (4) [below = of 3] {\bf \Large 4}; 
        \node[bver] (5) at ( 4,   0) {\bf \Large 5};        
        \node[bhor] (6) at ( 0.5,   -1.5) {\bf \Large 6};  
        \node[bver] (7) at ( 0,   -3) {\bf \Large 7};      
        \node[bhor] (8) [right = of 6] {\bf \Large 8};
\end{tikzpicture}  
}
\end{matrix}
~~~~~~~
T_3 = 
\begin{matrix}
\resizebox{!}{2.7 cm}{%
\begin{tikzpicture}[node distance=0 cm,outer sep = 0pt, line width=1pt]
        \node[bhor] (1) at ( 0,   0) {\bf \Large 0};
        \node[bhor] (2) [right = of 1] {\bf \Large 0};
        \node[bver] (3) at ( -0.5,   -1.5) {\bf \Large 2};
        \node[bhor] (4) at ( 1,   -1) {\bf \Large 2};     
        \node[bver] (5) [below = of 3] {\bf \Large 5};    
        \node[bhor] (6) [below = of 4] {\bf \Large 5};       
        \node[bver] (7) at ( 2.5,   -1.5) {\bf \Large 5};    
        \node[bver] (8) at ( 3.5,   -0.5) {\bf \Large 5};           
        \node[bhor] (9) [below = of 6] {\bf \Large 7};
\end{tikzpicture}  
}
\end{matrix}
$$
\end{center}
\caption{
\footnotesize
%Different types of tableaux.
%Standard Young tableau $T_0$ has shape $\la = (6,4,2,1,1)$ and descent set \{1,4,9,10,12\}.
Two standard domino tableaux $T_1$ and $T_2$ of shape $(5,5,4,1,1)$, descent set \{0,3,5,6\}, a semistandard tableau $T_3$ of shape $(5,5,4,3,1)$ and weight $\mu=(2,0,2,0,0,4,0,1)$. All of the tableaux have a spin of 2. 
}
\label{fig : DTs}
\end{figure}
%%%%%%%%%%%%%%%%%%%%%%%%%%%%%%%%%%%%%%%%%%%%%%%%%%%%%%%%%%
\subsection{Chow's type B quasisymmetric functions}
Chow defines in \cite{Cho01} another Type B extension of Gessel's algebra of quasisymmetric functions that is dual to Solomon's descent algebra of type B.\\
Let $X = \{\cdots,x_{\text{-}i},\cdots, x_{\text{-}1}, x_0, x_1,\cdots, x_i,\cdots\}$ be an alphabet of totally ordered commutative indeterminates  with the assumption that $x_{\text{-}i} = x_i$ and let $I$ be a subset of $\{0\} \cup [n-1]$, he defines a type B analogue of the fundamental quasisymmetric functions 
\begin{eqnarray*}
F_I^B(X) =  \sum_{\substack{0 = i_0\leq i_1\leq i_2\leq \ldots \leq i_n\\ j\in I \Rightarrow i_j < i_{j+1}}} x_{i_1}x_{i_2}\ldots x_{i_n}.
\end{eqnarray*}
Note the particular rôle of the variable $x_0$.\\
Given two signed permutations $\al \in B_n$ and $\beta \in B_m$, denote $\al \shuffle \beta$ the set of signed permutations in $B_{m+n}$ obtained by shuffling the letters of $\al$ and $\overline{\beta}$ where we shift the absolute value of the $\beta_i$ to get $\overline{\beta}$. Chow shows in \cite[Prop. 2.2.6]{Cho01}
\begin{equation}
F_{Des(\al)}^B(X)F_{Des(\beta)}^B(X) = \sum_{\gamma \in \al \shuffle \beta}F^B_{Des(\gamma)}(X).
\end{equation}
Furthermore, we introduce in \cite{MayVas17} modified generating functions for domino tableaux called domino functions. Given a semistandard domino tableau $T$ of weight $\mu$, denote $X^T$ the monomial $x_0^{\mu_0}x_1^{\mu_1}x_2^{\mu_2}\cdots$. For $\la \in {\mathcal P}^0(n)$ and an additional parameter $q$ we define the {\bf domino function} indexed by $\la$ on alphabet $X$ as
\begin{equation}
\mathcal{G}_\la(X;q)  = \sum\limits_{T \in SSDT(\la)}q^{sp(T)}{X}^{T}.
\end{equation}
\noindent These functions are related to Chow's quasisymmetric functions through
\begin{lem}[\cite{MayVas19}]
\label{lem : GdF}
For $\la \in {\mathcal P}^0(n)$, the $q$-domino function indexed by $\la$ can be expanded in the basis of Chow's quasisymmetric functions as
\begin{equation}
\label{eq : GqF}
\mathcal{G}_\la(X;q)  = \sum\limits_{T \in SDT(\la)}q^{sp(T)}F^B_{Des(T)}(X).
\end{equation}
\end{lem}
  
%%%%%%%%%%%%%%%%%%%%%%%%%%%%%%%%%%%%%%%%%%%%%%%%%%%%%%%%%%
\subsection{Type B Littlewood-Richardson coefficients}
There is a natural analogue of the RSK-correspondence for signed permutations involving {\bf domino tableaux}. 
Barbash and Vogan (\cite{BarVog82}) built a bijection between signed permutations of $B_n$ and pairs of standard domino tableaux of equal shape in $\mathcal{P}^0(n)$. An independent development on the subject is due to Garfinkle. Ta\c{s}kin (\cite[Prop. 26]{Tas12}) showed that the two standard domino tableaux $T$ and $U$ associated to a signed permutation $\pi$ by the algorithm of Barbash and Vogan have respective descent sets $Des(T) = Des(\pi^{-1})$ and $Des(U) = Des(\pi)$ while Shimozono and White showed in \cite{ShiWhi01} the {\bf colour-to-spin} property i.e. $
tc(\pi) = sp(T) + sp(U)$.\\
Let $X,Y$ be two alphabets of indeterminates. In \cite{MayVas19}, we show that Equation \eqref{eq : GqF} implies the following proposition.
\begin{prop}[\cite{MayVas19}]
\label{prop : CauchyDesB}
For any integer $n$, we have
\begin{align}
\label{eq : CauchyDesB}
\nonumber \sum_{\pi \in B_n}q^{tc(\pi)}&F^B_{Des(\pi)}(X)F^B_{Des(\pi^{-1})}(Y) =\\
& \sum_{\substack{\la \in {\mathcal P}^0(n)\\T,U \in SDT(\la)}}q^{sp(T)+sp(U)}F^B_{Des(T)}(X)F^B_{Des(U)}(Y).
\end{align}
\end{prop}
\noindent Proposition \ref{prop : CauchyDesB} provides a direct analytical proof of the properties showed in a bijective fashion by Barbash, Vogan, Ta\c{s}kin, Shimonozo and White. More formally it proves the existence for any standard domino tableau $T$ of a class of signed permutations $C^B_{T}$ such that for any $\pi \in C^B_T$, $Des(\pi^{-1}) = Des(T)$ and there is a descent preserving bijection $C^B_T \rightarrow SDT(\la), \pi \mapsto U$ such that $tc(\pi) = sp(T) + sp (U).$ In a similar fashion as in Section \ref{sec : prodSchur}, we use Equation \eqref{eq : CauchyDesB} to show some properties for the {\it type B Littlewood-Richardson coefficients}.\\

First we focus on the case $q=1$. There is a well known (not descent preserving) bijection between semistandard domino tableaux of weight $\mu$ and semistandard bi-tableaux of respective weights $\mu^-$ and $\mu^+$ such that $\mu^+_i + \mu^-_i = \mu_i$ for all $i$.
% (see~e.g. \cite[Algorithm 6.1]{CL95}). 
The respective shapes of the two Young tableaux depend only on the shape of the initial semistandard domino tableau. 
Denote $(T^-,T^+)$ the bi-tableau associated to a semistandard domino tableau $T$ of shape $\la$ and $(\la^-,\la^+)$ their respective shapes. $(T^-,T^+)$ (resp. $(\la^-,\la^+)$) is the {\bf $2$-quotient} of $T$ (resp. $\la$). Denote $X^* = X\setminus{x_0}$. For $\la \in \mathcal{P}^0(n)$ one has $$\mathcal{G}_\la(X;1) = s_{\la^-}(X^*)s_{\la^+}(X)$$ and for $\la, \mu \in \mathcal{P}^0(n)\times\mathcal{P}^0(m)$, 
\begin{equation}\label{eq : GGG}
\mathcal{G}_\la(X;1)\mathcal{G}_\mu(X;1) = \sum_{\nu \in \mathcal{P}^0(n+m)}c^{\nu^-}_{\la^-\mu^-}c^{\nu^+}_{\la^+\mu^+}\mathcal{G}_\nu(X;1).
\end{equation}
The following result extends Proposition \ref{prop : typeA} to signed permutations.
\begin{thm}
\label{thm : thmChow}
Let $\la, \mu \in \mathcal{P}^0(n)\times\mathcal{P}^0(m)$. For $I \subset [0]\cup[n+m-1]$ denote $D^B_I$ the set of signed permutations $\pi$ such that $Des(\pi) =I$ and $\Delta^B_I$ the set of standard domino tableaux of descent set $I$. Using the notations above for integer partitions $\la, \mu$ and $T,U$ in $SDT(\la)\times SDT(\mu)$
\begin{equation}
\label{eq : thmChow}
\sum_{\nu \in \mathcal{P}^0(n+m)}c^{\nu^-}_{\la^-\mu^-}c^{\nu^+}_{\la^+\mu^+}|SDT(\nu) \cap \Delta^B_I| =  \left|\left (C^B_T \shuffle C^B_U \right) \cap D^B_I\right|.
\end{equation}
\end{thm}
\begin{proof}(sketch)
Theorem \ref{thm : thmChow} is a consequence of Equations \eqref{eq : GGG} and \eqref{eq : FFF}.
\end{proof}
On the other hand, the general case for $q$ remains open. Indeed, using Proposition \ref{prop : CauchyDesB}, one can show for $\la \in \mathcal{P}^0(n)$ and $T \in SDT(\la)$ 
\begin{equation*}
\mathcal{G}_\la(X;q) = \sum_{\pi \in C^B_T}q^{tc(\pi) -sp(T)}F^B_{Des(\pi)}.
\end{equation*} 
Then, if $\mu \in \mathcal{P}^0(m)$ and $U \in SDT(\mu)$
\begin{equation*}
\mathcal{G}_\la(X;q)\mathcal{G}_\mu(X;q) = q^{-(sp(T)+sp(U))}\sum_{\gamma \in C^B_T \shuffle C^B_U}q^{tc(\gamma)}F^B_{Des(\gamma)}.
\end{equation*}
However, even if we showed that the set of signed permutations $C^B_T \shuffle C^B_U$ is in descent preserving bijection with a set of cardinality $c^{\nu^-}_{\la^-\mu^-}c^{\nu^+}_{\la^+\mu^+}|SDT(\nu)|$, we cannot say {\it a priori} how this bijection preserves or transforms the statistic $tc$. 
%%%%%%%%%%%%%%%%%%%%%%%%%%%%%%%%%%%%%%%%%%%%%%%%%%%%%%%%%%
%% if you use biblatex then this generates the bibliography
%% if you use some other method then remove this and do it your own way
\printbibliography

\end{document}